\documentclass[review,11pt]{elsarticle}

\usepackage[utf8]{inputenc}
\usepackage[T1]{fontenc}
\usepackage{enumerate,enumitem}
\usepackage{amsfonts,amsmath,amssymb,amsthm}
\usepackage{geometry,graphicx}
\usepackage{color}
\usepackage{dsfont}%do indykatora
\usepackage{fancyhdr}
\usepackage{mathrsfs, multirow}
\usepackage{subfigure}
\usepackage{url}
\usepackage{wrapfig}
\usepackage{yfonts}

\everymath{\displaystyle}

\usepackage[colorlinks=true]{hyperref}

\newtheorem{tw}{Theorem}[section]	
\newtheorem{corollary}[tw]{Corollary}

\newtheorem{proposition}[tw]{Proposition}

\newtheorem{lemma}[tw]{Lemma}
\theoremstyle{definition} 
\newtheorem{definition}[tw]{Definition}
\newtheorem{example}[tw]{Example}
\newtheorem{remark}[tw]{Remark}

\newproof{pf}{Proof}

\newcommand{\rA}{\mathrm{A}}
\newcommand{\rB}{\mathrm{B}}
\newcommand{\rC}{\mathrm{C}}

\newcommand{\rF}{\mathrm{F}}

\newcommand{\rK}{\mathrm{K}}

\newcommand{\rS}{\mathrm{S}}
\newcommand{\rT}{\mathrm{T}}

\newcommand{\rt}{\mathrm{t}}
\newcommand{\rs}{\mathrm{s}}

\newcommand{\ot}{\widehat{\rt}}
\newcommand{\os}{\widehat{\rs}}

\newcommand\bu{\mathbf{u}}

\newcommand\bx{\mathbf{x}}
\newcommand\bz{\mathbf{z}}

\newcommand\sI{\mathsf{I}}
\newcommand\sJ{\mathsf{J}}

\newcommand\sLi{\mathsf{Int}_{[0,1]}}

\newcommand{\adm}{\mathcal{S}_{\text{adm}}}

\newcommand\mR{{\mathbb R}}
\newcommand\moR{\overline{\mathbb R}}

\renewcommand\ge{\geqslant}
\renewcommand\le{\leqslant}

\newcommand\leqs{\leqa_{\mathrm{Int}}}
\newcommand\leqa{\preceq}

\newcommand{\tx}{\textstyle{}}

\usepackage{xcolor}\definecolor{darkgreen}{rgb}{0,0.5,0}

\begin{document}
\begin{frontmatter}

\title{
On admissible pairs of aggregation functions based on quasi-linear means and related families
}

\author{Micha{\l} Boczek\fnref{label1}}
\ead{michal.boczek.1@p.lodz.pl}

\author{Marek Kaluszka\fnref{label1}}
\ead{marek.kaluszka@p.lodz.pl}

\author{Jakub {\L}ompie\'{s}\corref{cor1}\fnref{label1}}
\ead{jakub.lompies@p.lodz.pl}
\cortext[cor1]{jakub.lompies@p.lodz.pl}

\address[label1]{Institute of Mathematics, Lodz University of Technology, 93-590 Lodz, Poland}

\begin{abstract} 
Admissible orders play a key role in ranking subintervals of the unit interval.
%When comparing closed subintervals of $[0,1]$, admissible orders play a key role in ranking interval pairs. 
In 2013, Bustince et al.~proposed constructing such relations by means of admissible pairs of aggregation functions. 
The only significant example in the literature is a~pair of weighted arithmetic means with different weights.
In this paper, we present a~method for constructing admissible pairs of aggregation functions, which allows us to verify the admissibility of various function classes, including quasi-linear means, Archimedean $t$-norms (and $t$-conorms), and certain strictly Schur-convex (or Schur-concave) functions. Furthermore, we examine the relationship between admissible orders generated by admissible pairs of aggregation functions and the $(\alpha, \beta)$-order, identifying cases where these two notions do not coincide.
\end{abstract}

\begin{keyword}
Admissible order\sep 
Admissible pair of aggregation functions\sep 
Aggregation function \sep 
Archimedean t-norm \sep
Quasi-linear mean\sep
Ranking for intervals.
\end{keyword}
\end{frontmatter}

\section{Introduction}\label{sec:1}

The study of admissible orders on the family of closed subintervals of $[0,1]$, understood as total orders refining the interval order, 
%— also known as the product order or the Kulisch–Miranker order — 
began with the innovative work of Bustince et al.~\cite{bustince2013}.
These total orders have been applied in various areas, such as decision-making \cite{bentkowska2015, bustince2013b, pekala2019, wu2024}, classification \cite{derrac2016, takac2022}, and image processing \cite{bustince2009}, where ranking intervals is required.
A~main property of an admissible order is that it can be constructed using an aggregation function (AF, in short) of the endpoints of the intervals
~\cite{bustince2013, gupta2023, takac2022}. 
Based on this approach, numerous constructions of interval-valued operators with practical applications have been proposed in the literature, including:
interval-valued aggregation function \cite{BKL25,bustince2020, gupta2023}, 
interval-valued implication \cite{zapata2017}, 
interval \textbf{R}-Sheffer strokes \cite{yzhao2024}, 
interval-valued negation function \cite{asiain2018},
interval-valued fuzzy logic \cite{he2023}, interval-valued seminormed fuzzy integral \cite{BJK21}, and 
interval-valued Choquet integral \cite{BJK22, takac2022}.

A~pair of AFs is admissible if and only if a~specific requirement holds (see condition~\ref{Adm} in Proposition \ref{prop:adm}).
All known 
admissible pairs of AFs in the literature are  indistinguishable from the pair of weighted arithmetic means with different weights under condition~\ref{Adm}. %Dużo technicznych szczegółów jak na Wstep
Although such pairs have received considerable attention, their limitations in flexibility make the search for structurally new admissible pairs particularly important. 
This is 
%particularly
especially 
relevant in the context of interval-valued operators, which are gaining prominence in decision-making, machine learning, and information fusion.

This paper addresses this issue by presenting several new examples of admissible pairs of AFs. 
We focus on widely used AFs such as the quasi-linear means (applied in machine learning),   Archimedean $t$-norms, Archimedean $t$-conorms (frequently used in the theory of fuzzy sets), and some strictly Schur-convex/concave functions (important in statistical applications).
A~secondary aim of this paper is to identify an admissible pair of AFs that generates a~total order distinct from the $(\alpha, \beta)$-order.
This question is motivated by both theoretical   and practical considerations in applications, where the $(\alpha, \beta)$-order may not capture the desired ordering behavior.

The paper is organized as follows. 
In Section~\ref{sec:prel}, we introduce the preliminary definitions and concepts. 
Section~\ref{sec:aux} presents a method to construct
an admissible pair of AFs, which is used in Section~\ref{sec:main} to determine admissibility for some pairs of AFs.
%auxiliary results that are crucial for the determination of numerous examples of admissible pairs of AFs, which are presented in Section~\ref{sec:main}. 
Finally, in Section~\ref{sec:coin},  we include several results concerning the relationship between admissible orders and $(\alpha, \beta)$-orders, and provide an example of a~total order generated by an admissible pair of AFs that does not coincide with any $(\alpha, \beta)$-order.

\section{Background and notation}\label{sec:prel}

Throughout the paper,  
$\sLi=\{[a_1, a_2]\mid 0\le a_1\le a_2\le 1\}$. 
Intervals from $\sLi$ will be represented by bold letters, while their left and right endpoints will be denoted by the corresponding non-bold letters with subscripts 1 and 2, respectively, e.g., $\bx= [x_1, x_2].$ 
Set $a\wedge b = \min\{a,b\}$ and $a\vee b = \max\{a,b\}$ for $a,b\ge 0$. 
%%%%
%{\mb Note that that set $\sLi$ is equipollent   to the set $\mathsf{K}_{[0,1]} = \{(x_1, x_2) \mid x_1 \le x_2\}$ \cite{bustince2013,moore1982,rubin1967}. Therefore, bold letters will be also used to denote points belonging to the set $\mathsf{K}_{[0,1]}$.}
%%%

\subsection{\textbf{Admissible order}}\label{sec:2.1}

The natural partial order on $\sLi$  is the \textit{interval order} ($\leqs$)  given by
\begin{align*}
 \bu\leqs \bx\quad \Leftrightarrow \quad  u_1\le x_1 \text{ and } u_2\le x_2.
\end{align*}
In numerous situations, refining the interval order into a~total order is necessary. To address this, Bustince et al. \cite{bustince2013} introduced the concept of an admissible order on $\sLi$.

\begin{definition}
An \textit{admissible order} 
%\sout{(on $\sLi$)} %jeśli definiujemy ,,dopuszczalny porządek" to z definicji wiadomo, że jest na $\sLi$
%wziąłem w nawias, aby potem nie pisać w kółko "admissible order on $\sLi$'
is a~total order $\leqa$ on $\sLi$ such that it refines the interval order, i.e., for each $\bu,\bx\in \sLi,$ it holds $\bu\leqa \bx$ whenever $\bu \leqs \bx$. 
\end{definition}

An~important admissible order  is the \textit{$(\alpha, \beta)$-order} on  $\sLi$, defined as follows 
%\sout{for any $\bu,\bx \in \sLi$}
\begin{align}\label{aborder}
\bu \leqa_{(\alpha, \beta)} \bx \qquad \Leftrightarrow \qquad 
\begin{cases}
    \rK_{\alpha}(\bu) < \rK_{\alpha}(\bx), \text{ or } \\
    \rK_{\alpha}(\bu) = \rK_{\alpha}(\bx) \text{ and } \rK_{\beta}(\bu) \le \rK_{\beta}(\bx),
\end{cases}
\end{align}
where $\alpha, \beta \in [0, 1]$, $\alpha \neq \beta$,
and $\rK_{w}(\bz) =(1- w) z_1 + w z_2$   for $\bz=[z_1,z_2]\in \sLi$.\footnote{The $(\alpha,\beta)$-order  generalizes the following well-known admissible orders: 
\textit{lexicographical order}  for  $(\alpha,\beta)=(0,1)$, \textit{antilexicographical order} 
for $(\alpha,\beta)=(1,0)$,  
\textit{Xu-Yager order} for $(\alpha,\beta)=(0.5, 1)$ \cite{xu2006}, and \textit{information
quality order} for $(\alpha,\beta)=(0.5, 0)$  \cite{asmus2022}.}
By replacing $\rK_{\alpha}$ and $\rK_{\beta}$ in~\eqref{aborder} with appropriate and different aggregation functions (see Definition~\ref{def:ag}), Bustince et al.~\cite{bustince2013} proposed a~universal method for constructing an admissible order based on the aggregation of interval endpoints.
In the next section, we recall this method, which plays a~crucial role in the context of the present work.

\begin{remark} 
Further examples of admissible orders, possibly not defined on $\sLi$, have been presented in \cite{asmus2022, BJK21, miguel2016, milfont2021, santana2020, sussner2024, sussner2025, zhang2025, zumelzu2022}. 
It is worth noting that~\cite[Sec.~2]{BJK21}
presents a~specific characterization of admissible orders on $\sLi$.
\end{remark}

\subsection{\textbf{The relation generated by a~pair of aggregation functions}}\label{sec:2.2}

\begin{definition}\label{def:ag}
A~function $\rA\colon \sLi\to [0,1]$ is called an~\textit{aggregation function} (AF, in short)  
on $\sLi$ 
if $\rA([0,0]) = 0,$ $\rA([1,1]) = 1$, and $\rA$ is increasing, i.e., $\rA(\bu) \le \rA(\bx)$ whenever $u_1\le x_1$ and $u_2 \le x_2$.   
\end{definition}

Due to the existence of a~bijection between the spaces $\sLi$ and $\mathsf{K}_{[0,1]} = \{(x,y)\in [0,1]^2 \mid x\le y\}$, the definition of an AF on $\sLi$ is equivalent to that on $\mathsf{K}_{[0,1]}$. 
Therefore, for simplicity, we often write
$\rA(u_1, u_2)$ instead of $\rA([u_1, u_2])$. 
Moreover, we write ``AF'' instead of ``AF on $\sLi$''.

\begin{remark}
Note that the concept of AF proposed in Definition~\ref{def:ag} is not the standard one commonly found in the literature \cite{alsina2006, bustince2013, grabisch2009, klement2000}. 
The only difference is that it is defined on $\sLi$ (equivalently on $\mathsf{K}_{[0,1]}$) instead of $[0,1]^2$, while the rest of the conditions remain consistent with the classical definitions. 
This adjustment is made to highlight the properties that are crucial and employed in the subsequent analysis.
\end{remark}

We say that the binary relation $\leqa_{\rA,\rB}$ is \textit{generated by a~pair}  $(\rA,\rB)$ of AFs
if and only if for all $\bu,\bx\in \sLi$,
\begin{align*}
    \bu\leqa_{\rA,\rB} \bx \qquad \Leftrightarrow \qquad 
\begin{cases}
    \rA(\bu) < \rA(\bx), \text{ or } \\
    \rA(\bu) = \rA(\bx) \text{ and } \rB(\bu) \le \rB(\bx).
\end{cases}
\end{align*}

\begin{definition}\label{def:order}(cf. \cite{bustince2013})
A~pair  $(\rA,\rB)$  of AFs  is said to be {\it admissible}  if $\leqa_{\rA,\rB}$ is an admissible order.
The set of all admissible pairs of AFs is denoted by $\adm$. 
\end{definition}

\begin{proposition}\label{prop:adm}
A~pair $(\rA,\rB)$ belongs to $\adm$ 
if and only if the following condition is valid
\begin{enumerate}[leftmargin = 3.2em, label = (Adm)]
    \item for all $\bu,\bx\in \sLi$, the equalities $\rA(\bu) =\rA(\bx) $ and $\rB(\bu)= \rB(\bx)$ can only hold if $\bu=\bx$.\label{Adm}
\end{enumerate}
\end{proposition}
\begin{proof}
    It is easy to check that the relation $\leqa_{\rA, \rB}$ is reflexive, transitive, and strongly connected.
Moreover, it is antisymmetric if and only if condition~\ref{Adm} holds.
\end{proof}

\begin{remark}
In  \cite{bustince2013}, the authors additionally assumed the continuity of the AFs, although this is not necessary for the admissibility of the relation $\leqa_{\rA,\rB}$. 
This assumption is absent in the later literature as well \cite{bustince2013b, takac2022}.
\end{remark}

Based on Proposition \ref{prop:adm}, the admissibility of the pair $(\rA,\rB)$ of AFs guarantees that $(\rB,\rA)$ is also admissible.
To the best of our knowledge, the only admissible pair of AFs known in the literature is $(\rK_{\alpha}, \rK_{\beta})$  for $\alpha\neq \beta$.
Some recent results also provide examples of non-admissible pairs \cite{gupta2023}.
%% NIE USUWAC ODPOWIEDZ NA EWENTUALNY ZARUZT RECENZENTA
%%\footnote{As mentioned in Section~\ref{sec:1}, the classical definition of an AF in the literature considers binary functions defined on the set $[0,1]^2$.  In this context, the literature provides examples of admissible pairs of AFs differs from $(\overline{\rK}_{\alpha}, \overline{\rK}_{\beta})$ such as $(\wedge|_{[0,1]^2}, \vee|_{[0,1]^2})$, where $\overline{\rK}_w(x,y) = (1-w)x + w y$ for any $w,x,y\in [0,1]$ \cite{gupta2023}. However, in this paper, we restrict the domain of an AF to $\sLi$, as this is sufficient for verifying property \ref{Adm}. Consequently $(\wedge|_{\sLi}, \vee|_{\sLi})$ is the same as $(\overline{\rK}_0|_{\sLi}, \overline{\rK}_1|_{\sLi})$ under condition~\ref{Adm}.}}
In the next example, we present additional cases of admissible pairs of AFs.

\begin{example}\label{ex:2.6}
Let $\rB$ be an AF.
Then, $(\rK_0,\rB)\in \adm$ if and only if  the function $[x_1, 1]\ni x\mapsto \rB(x_1, x)$ is strictly increasing for any $x_1$. 
Moreover, $(\rK_1, \rB)\in \adm$ if and only if the function $[0, x_2]\ni x\mapsto \rB(x, x_2)$ is strictly increasing  for any $x_2$.
\end{example}

As far as we know,  the admissibility of pairs of AFs, other than those in Example~\ref{ex:2.6} and $(\rK_{\alpha}, \rK_{\beta})$, has not yet been examined.
We focus on the admissibility of pairs from several classes of AFs commonly discussed in the literature, including quasi-linear means (Sec.~\ref{sec:quasi}), Archimedean $t$-norms, Archimedean $t$-conorms (Sec.~\ref{sec:arch}), and certain strictly Schur-convex or Schur-concave functions (Sec.~\ref{sec:schur}).
Before proceeding with a~detailed discussion, we first analyze the validity of  condition \ref{adm2}, given by
\begin{enumerate}[leftmargin = 4em, label = (Adm2)]
    \item \label{adm2} for any $s_1,s_2,t_1,t_2 \in \sI$ such that $s_1 \le s_2$ and $t_1 \le t_2,$  the equalities $\rK_{v_1}(s_1, s_2) = \rK_{v_1}(t_1, t_2)$ and $\rK_{v_2}(h(s_1), h(s_2)) = \rK_{v_2}(h(t_1), h(t_2))$ can only hold if $s_1 = t_1$ and $s_2 = t_2$.
\end{enumerate}
This condition follows from the property~\ref{Adm} for the aforementioned pairs of AFs.
%which follows as a~consequence of property~\ref{Adm} for the aforementioned pairs of AFs.
This will be the topic of the next section.

\section%{\sout{Preliminary results} 
{Analytical tools for identifying admissible pairs of AFs}\label{sec:aux}
%tak sekcja brzmi bardziej powaznie a nie tak prymitywnie

From now on, 
$\sI$ is any non-degenerate interval in $\mR$ and $\sI^2_{<} = \{(t_1, t_2)\in \sI^2\mid  t_1 <t_2\}$.

\begin{lemma}\label{lem:wr}
Let $h\colon \sI \to \mR$ and
$v_1,v_2 \in (0,1)$. 
Condition~\ref{adm2}
is satisfied  if and only if 
$H(x, t_1, t_2) \neq 0$
for any $x\in (0, t_2 - t_1]$  and  any $(t_1,t_2) \in\sI^2_{<}$,
where for $x\in [0,t_2-t_1]$,
\begin{align}\label{ap:n1}
    H(x, t_1, t_2) = (1-v_2)\big(h(t_1 + v_1 x)-h(t_1)\big) +v_2\big(h(t_2 - (1-v_1)x)-h(t_2)\big).\; %x\in [0,t_2-t_1].
\end{align}
Moreover, if the function $h$ is convex (resp., strictly convex, concave, or strictly concave), then $H$ is also convex (resp., strictly convex, concave, or strictly concave) in the first argument for any $(t_1,t_2)\in \sI^2_{<}$.
\end{lemma}
\begin{proof}
The following statements are equivalent:
\begin{enumerate}[noitemsep, label =(\alph*)]
    \item Condition~\ref{adm2} \textit{does not} hold;

    \item  There are $s_1\le s_2$ and $t_1\le t_2$ such that  $\rK_{v_1}(s_1, s_2) = \rK_{v_1}(t_1, t_2),$ and $\rK_{v_2}(h(s_1), h(s_2)) = \rK_{v_2}(h(t_1), h(t_2)),$ and  ($s_1 \neq t_1$ or $s_2 \neq t_2$);

    \item There are $s_1\le s_2$ and $t_1\le t_2$ such that  $\rK_{v_1}(s_1, s_2) = \rK_{v_1}(t_1, t_2),$ and $\rK_{v_2}(h(s_1), h(s_2)) = \rK_{v_2}(h(t_1), h(t_2)),$ and ((i) $t_1<s_1$ and $s_2 < t_2,$ or (ii) $s_1 < t_1$ and $t_2 < s_2$).
\end{enumerate}
The equivalence of (b) and (c) is due to the fact that, the condition $\rK_{v_1}(s_1, s_2) = \rK_{v_1}(t_1, t_2)$ (which is the same as  $(1-v_1)(s_1-t_1) = v_1(t_2-s_2)$) can hold in one of the following cases: (i) $t_1 < s_1 \le s_2 < t_2$ and (ii)  $s_1 <t_1 \le t_2 < s_2$.

For case (i), we set $x=(s_1-t_1)/v_1$ in point (c), where $0<x\le t_2-t_1,$ 
so $s_1=t_1+v_1x$ and $s_2=t_2-(1-v_1)x,$ as  $\rK_{v_1}(s_1, s_2) = \rK_{v_1}(t_1, t_2).$
For case (ii), we put $y=(t_1-s_1)/v_1$ in point (c) with $0 < y \le s_2-s_1$, so $t_1=s_1+v_1 y$ and $t_2=s_2-(1-v_1)y.$
Condition~\ref{adm2} \textit{does not} hold if and only if there exist $(s_1,s_2),(t_1,t_2)\in \sI^2_{<}$, $x\in (0,t_2-t_1],$ and $y\in (0,s_2-s_1]$ such that
$H(x,t_1,t_2)=0$ or $H(y,s_1,s_2)=0,$ as desired.
\end{proof}

Applying Lemma \ref{lem:wr}, we  identify  functions $h$ for which condition~\ref{adm2} holds.

\begin{tw}\label{tw:main}
Let $h\colon \mathsf{I} \to \mR$ 
and $v_1,v_2 \in (0,1)$.

\begin{enumerate}[noitemsep, label=(\alph*)] 
    \item 
    Assume that $v_1 = v_2$ and $h$ is a~continuous function. 
    Condition~\ref{adm2} is satisfied if and only if $h$ is strictly convex or strictly concave. 
\end{enumerate}
Condition~\ref{adm2} also holds for $h(x) = ax + b$ with $a,b\in \mR$ if and only if $a\neq 0$ and $v_1 \neq v_2$. 
Moreover, condition~\ref{adm2} is true
if one of the following conditions is satisfied: 
\begin{enumerate}[noitemsep, label=(\alph*), start = 2]
    \item  $v_1 < v_2,$ and  $h$ or $-h$ is a~convex and strictly increasing function;
    
    \item  $v_1 > v_2$,  and $h$ or $-h$ is a~convex and strictly decreasing function.
\end{enumerate}
\end{tw}
\begin{proof} 
Let $H$ be defined as in \eqref{ap:n1}.
To simplify notation, we set $U(t_1, t_2) = H(t_2 - t_1, t_1, t_2)$ for any $(t_1,t_2)\in \sI^2_{<}$. 
We first assume that $v = v_1 = v_2$ and  the function $h$ is strictly convex or strictly concave. Clearly, 
\begin{align}\label{ap:n2}
    U(t_1,t_2) = h((1-v)t_1 + v t_2) - (1-v)h(t_1) - vh(t_2),\qquad (t_1,t_2) \in\sI^2_{<}.
\end{align}
For each $(t_1,t_2) \in\sI^2_{<}$,
we have 
$U(t_1, t_2) <0$
(resp., $U(t_1, t_2) >0$) if $h$ is strictly convex (resp., strictly concave). 
Combining the definition of strict convexity (resp., strict concavity)  of 
$H(\cdot, t_1, t_2)$ (see Lemma~\ref{lem:wr}) with  $H(0, t_1, t_2) = 0$, we get 
\begin{align}\label{eq:convex}
H(x, t_1, t_2) \le  \frac{U(t_2- t_1)}{t_2-t_1} x  < 0 \quad \Big(\text{resp., } H(x, t_1, t_2)  \ge \frac{U(t_2- t_1)}{t_2-t_1} x >0\Big)
\end{align}
for any  $x\in (0,t_2-t_1]$. 
Hence, 
%} \sout{As a~result, we conclude that}
$H(x, t_1, t_2)\neq 0$ for any $x\in (0, t_2-t_1]$ and any $(t_1,t_2) \in \sI^2_{<}$ and, by Lemma~\ref{lem:wr},
%Thus, by Lemma~\ref{lem:wr},  
condition~\ref{adm2} holds true.

We now show the converse of the implication stated in point (a). 
Let $v=v_1 = v_2$.
Assume that condition~\ref{adm2} holds. 
From Lemma~\ref{lem:wr}, 
%\sout{it follows that} 
$U(t_1, t_2)\neq 0$ for any $(t_1,t_2) \in\sI^2_{<}$.
We consider three cases.
\begin{enumerate}[noitemsep, label = (a\arabic*)]
    \item Let 
    $U(t_1, t_2) <0$
    for any $(t_1,t_2) \in\sI^2_{<}$. By \eqref{ap:n2}, $h((1-v)t_1 + v t_2) < (1-v)h(t_1) + vh(t_2)$ for any $(t_1,t_2) \in\sI^2_{<}$. 
    The Dar\'{o}czy-P\'{a}les identity
    %pojęcie to widnieje w literaturze. Brzmi naukowo, więc może warto takie coś zastosować
    $$
    \frac{t_1 + t_2}{2} = v\Big(v \frac{t_1 + t_2}{2} + (1-v)t_1\Big) + (1-v)\Big(vt_2 + (1-v)\frac{t_1 + t_2}{2}\Big),
    $$
    can be used to show that $h$ is  strictly mid-convex, that is,   $h(0.5(t_1+t_2)) < 0.5h(t_1) + 0.5h(t_2)$ 
    (cf.~\cite[Lemma~1]{daroczy1987}). 
    This condition is equivalent to the strict convexity of $h$ (see \cite[Sec.~1.1]{niculescu2018}).

    \item If 
    $U(t_1, t_2) > 0$
    for any $(t_1,t_2) \in\sI^2_{<}$, then, by 
    reasoning analogous to that in point (a1), $h$ is strictly concave.

    \item Assume that 
    $U(t_1, t_2) > 0$
    and 
    $U(t_1^{\ast}, t_2^{\ast}) <0$ for some $(t_1, t_2),(t_1^{\ast},t_2^{\ast})\in \sI^2_{<}$.
    Without loss of generality, we assume that $t_1 \le  t_1^\ast.$ If $t_1 <  t_1^\ast,$ we put 
    $$
    y(x)=\frac{t_2^\ast-t_2}{t_1^\ast-t_1}(x - t_1) + t_2\quad \text{and}\quad g(x)=h((1-v)x+vy(x))-(1-v)h(x)-vh(y(x)),
    $$
    where $x\in [t_1, t_1^{\ast}]$ and $g$  is continuous.
    Then, by \eqref{ap:n2},
    $g(t_1) = U(t_1, t_2)>0$ and 
    $g(t_1^\ast) = U(t_1^{\ast}, t_2^{\ast})<0$, 
    so by the intermediate value theorem 
    there exists 
    $\overline{t}\in (t_1, t_1^\ast)$ such that $g(\overline{t})=0$.
    Since $t_1<y(t_1),$ $t_1^\ast<y(t_1^\ast)$, and $\overline {t}=\alpha t_1+(1-\alpha)t_1^\ast$ for some $\alpha\in (0,1)$, we have $\overline{t}<\alpha  y(t_1)+(1-\alpha)y(t_1^\ast)=y(\overline{t})$. 
    This implies that 
    $U(\overline{t}, y(\overline{t})) = g(\overline{t})= 0$, 
    as $(\overline{t}, y(\overline{t})) \in \sI^2_{<}$, which gives  a~contradiction.
    
    Now, assume that $t_1 = t_1^{\ast}$. Put 
    \begin{align*}
        y(x)=(t_2^\ast-t_2)x+t_2\quad\text{and} \quad g(x)=h((1-v)t_1+vy(x))-(1-v)h(t_1)-vh(y(x)),
    \end{align*}
    where $x\in [0,1]$. 
    Then, by \eqref{ap:n2}, 
    $g(0) = U(t_1, t_2)>0$ 
    and 
    $g(1) = U(t_1^{\ast}, t_2^{\ast}) < 0$, 
    so the intermediate value theorem implies that $g(\overline{t})=0$ for some $\overline{t}\in (0,1).$
    %\sout{there exists $\overline{t}\in (0,1)$ such that $g(\overline{t})=0.$}
    Clearly, $y(\overline{t})$ belongs to the open interval with endpoints $y(0)=t_2$ and $y(1)=t_2^\ast.$  
    Since $t_1 = t_1^{\ast} < \min\{t_2,t_2^\ast\}< y(\overline{t})$, we have $(t_1, y(\overline{t})) \in\sI^2_{<}$ and 
    $U(t_1, y(\overline{t})) = g(\overline{t}) =0$,   a~contradiction.  
    Point (a) has been established.
\end{enumerate}

For $h(x)=ax+b$ for any $x\in \sI$, we have $H(x,t_1,t_2) = a(v_1 -v_2)x$  for any $x,t_1,t_2.$
Consequently, condition~\ref{adm2} is satisfied if and only if $a\neq 0$ and $v_1 \neq v_2$.

Statements (b) and (c) follow from the formula 
\begin{align*}
    U(t_1, t_2) =h((1-v_1)t_1 +v_1 t_2) - (1-v_1)h(t_1)-v_1 h(t_2) + (v_2-v_1)(h(t_1) - h(t_2)).
\end{align*} 
For example, if $v_1 <v_2$ and $h$ is convex and strictly increasing (resp., $h$ is concave and strictly decreasing), then $U(t_1, t_2) < 0$ (resp., $U(t_1, t_2)>0$) for any $(t_1,t_2)\in \sI^2_{<}$.
Since $H(0, t_1, t_2)=0$, by convexity of $H(\cdot, t_1, t_2)$ (resp., concavity of $H(\cdot, t_1, t_2)$) for any $(t_1, t_2) \in \sI^2_{<}$,  we get \eqref{eq:convex}
for any  $x\in (0,t_2-t_1]$.
Hence, 
$H(x, t_1, t_2) < 0$ (resp., $H(x, t_1, t_2) > 0$) for all $x\in (0, t_2 - t_1]$ and   $(t_1, t_2) \in \sI^2_{<}$.
Using Lemma~\ref{lem:wr}, we get the statement. 
\end{proof}

\medskip

In what follows, we show
that the converse of the implication in point (b) of Theorem~\ref{tw:main} does not hold, that is, for some $v_1<v_2$ there exists  a~strictly concave and not strictly decreasing function $h\colon \sI\to \mR$ such that condition~\ref{adm2} holds.

\begin{example}
Let  $\sI = [a,b]\subset \mR$, $a<b$, and $h\colon \sI \to \mR$ be a~strictly concave
and differentiable function 
(with finite one-sided derivatives at the endpoints of $\sI$) 
and $h'(b)>0$. 
Then the function $[0, t_2 - t_1] \ni x \mapsto H(x, t_1, t_2)$ is strictly concave for any $(t_1,t_2)\in \sI^2_{<}$ (cf.~Lemma~\ref{lem:wr}). 
Thus, the Stolz Theorem (cf. \cite[Sec.~1.4.1]{niculescu2018}) yields
$$
H(x,t_1, t_2)\le H(0,t_1,t_2) + x\cdot \frac{\partial }{\partial y} H(y, t_1, t_2)|_{y=0}
$$
for all $0<x \le t_2-t_1.$
Since $H(0, t_1, t_2)=0$, by Lemma~\ref{lem:wr}, 
condition~\ref{adm2} holds 
 if 
${\tx \frac{\partial }{\partial y} H(y, t_1, t_2)|_{y=0} < 0}$ 
for any $(t_1,t_2) \in \sI^2_{<}$, which can be rewritten as 
\begin{align}\label{ex:eq1}
    h'(t_1)<\frac{(1-v_1)v_2}{(1-v_2)v_1}h'(t_2) \qquad \text{ for any } (t_1,t_2) \in \sI^2_{<}.
\end{align}
The function $h'$ is strictly decreasing  on $\sI$, 
%%Theorem B Roberts 1976
so 
condition \eqref{ex:eq1} is equivalent to
\begin{align}\label{ex:eq2}
    h'(a) < \frac{(1-v_1)v_2}{(1-v_2)v_1}h'(b).   
\end{align}
As  $h'(b)>0$, it is always possible to choose $v_1,v_2$  such that $v_1 < v_2$ and the fraction in \eqref{ex:eq2} becomes arbitrarily large, ensuring that inequality \eqref{ex:eq2} will be satisfied.
For instance, for $\sI = [1,2]$, $v_2 = 1-v_1 = 0.6$, and increasing and strictly concave function 
$h(x) = \sqrt{x}$, condition \eqref{ex:eq2} is valid, as  $0.5=h'(1)< 2.25\, h'(2)\approx 0.7955$.
\end{example}

We now restrict our attention to strictly monotone functions, under which the following characterization holds.
% nie zmieniaj...barwnośc jezyka pomoze nam poznie w antyplagiacie
\begin{proposition}\label{pro:3.5}
Let $h\colon \sI \to \mR$.
\begin{enumerate}[noitemsep, label=(\alph*)]
    \item If $h$ is strictly increasing, then condition \ref{adm2} holds for any $v_1,v_2\in (0,1)$ such that $v_1 < v_2$ (resp., $v_1 > v_2$) if and only if $h$ is convex (resp., concave).

    \item If $h$ is strictly decreasing, then  condition \ref{adm2} holds for any $v_1,v_2\in (0,1)$ such that $v_1 < v_2$ (resp., $v_1 > v_2$) if and only if $h$ is concave (resp., convex).
\end{enumerate}    
\end{proposition}
\begin{proof}
Let $h$ be a~strictly increasing function.
\begin{enumerate}[noitemsep, label = (a\arabic*)]
    \item From Theorem~\ref{tw:main}\,(b) we see that, for any $v_1 <v_2$, the convexity of $h$ implies that condition~\ref{adm2} holds. 
     This can also be stated as: the convexity of $h$ implies  that condition~\ref{adm2} holds for any $v_1<v_2$.

    \item 
     Assume that  condition~\ref{adm2} is true for any $v_1 < v_2$.
    Suppose that $h$ is not convex. Then $F(v_1) < 0$ for some  $(t_1, t_2)\in \sI^2_<$ and $v_1 \in (0,1)$, where 
    $$
    F(x)=(1-x)h(t_1) + xh(t_2)-h((1-v_1)t_1+v_1t_2), \quad x\in [v_1,1].
    $$ 
    Since $h$ is strictly increasing, we have $F(1)=h(t_2)-h((1-v_1)t_1+v_1t_2) > 0$.  
    By the intermediate value theorem, there exists $v_2\in (v_1, 1)$ such that  $F(v_2)=0$.
    This implies that $H(t_2-t_1, t_1, t_2)=F(v_2)= 0$, where $H$ is defined in \eqref{ap:n1}. By Lemma~\ref{lem:wr} we get, a~contradiction. 
    Thus, $h$ is convex.
%%NIE USUWAĆ
%%PRZYPADEK DLA $v_1 > v_2$ w punkcie (a)    \item  Assume that  condition~\ref{adm2} is true for any $v_1 > v_2$.     Suppose that $h$ is not concave.     Then $F(v_1) < 0$ for some  $(t_1, t_2)\in \sI^2_<$ and $v_1 \in (0,1)$, where     $$ F(x)=(1-x)h(t_1) + xh(t_2)-h((1-v_1)t_1+v_1t_2), \quad x\in [0, v_1].  $$   Since $h$ is strictly increasing, we have $F(0)=h(t_1)-h((1-v_1)t_1+v_1t_2) < 0$. By the intermediate value theorem, there exists $v_2\in (0, v_1)$ such that  $F(v_2)=0$. This implies that $H(t_2-t_1, t_1, t_2)={\mb F(v_2)=} 0$, {\mb where $H$ be defined in \eqref{ap:n1}. By Lemma~\ref{lem:wr} we get}, a~contradiction. Thus, $h$ is concave.}
\end{enumerate}
Combining the results  from (a1) and (a2), we establish point (a) for  $v_1 < v_2$. 
A~similar argument applies for $v_1 > v_2$, with Theorem~\ref{tw:main}\,(c) instead of Theorem~\ref{tw:main}\,(b).
Finally, point (b) follows from point (a) applied to the function $-h.$
\end{proof}

\section{Admissibility of pairs of aggregation functions}
\label{sec:main}

Based on the results obtained in Section~\ref{sec:aux}, 
 we now turn to the main contributions of this paper,
%we now address the main objective of this paper,
namely, the investigation of the admissibility of certain classes of AFs. These include quasi-linear means, Archimedean $t$-norms, Archimedean $t$-conorms, as well as certain strictly Schur-convex or Schur-concave functions.

\subsection{\textbf{Quasi-linear means}}\label{sec:quasi}

For a~fixed $w\in (0,1)$ and a~strictly monotone continuous function $f\colon[0,1]\to\moR$,  
the \textit{quasi-linear} (or \textit{weighted quasi-arithmetic}) \textit{mean  generated by $f$}
is the function $\rK^f_w\colon \sLi \to [0,1]$ defined by 
%%%OZNACZENIA INNE:
%w Beliakovie mamy takie oznaczenia $\mathrm{M}_{w,f}$
$\rK_{w}^{f}(\bz)= f^{-1}((1-w)f(z_1) + wf(z_2))$ under the convention that $-\infty +\infty=\infty -\infty=-\infty$, where $f^{-1}$ is continuous on 
$\operatorname{Ran}(f)$, the range of $f$.

In the case of pairs of AFs consisting of quasi-linear means, the values of their generating functions at the endpoints of the domain $[0,1]$ play a~crucial role in determining their admissibility.

\begin{tw}\label{tw:quasi_not} 
$(\rK_{w_1}^{f}, \rK_{w_2}^{g}) \notin \adm$ if $|f(0)|\wedge |g(0)|=\infty$ or $|f(1)|\wedge |g(1)|=\infty$.
\end{tw}
\begin{proof}
Due to the monotonicity of the generator function, 
$\rK_{w_1}^{f}(0,x)=0=\rK_{w_2}^{g}(0,x)$ for $x\in (0,1)$ if (i) $|f(0)|\wedge |g(0)|=\infty$ and 
$\rK_{w_1}^{f}(x,1) = 1 = \rK_{w_2}^{g}(x,1)$ for $x\in (0,1)$ if (ii) $|f(1)|\wedge |g(1)|=\infty.$ 
For $\bu  = [0,u_2]$ and $\bx = [0,x_2]$ with $0 < u_2 < x_2<1$ in case (i) and for $\bu  = [u_1, 1]$ and $\bx = [x_1,1]$  with $0<u_1<x_1<1$ in case (ii), we get $\rK_{w_1}^{f}(\bu)=\rK_{w_1}^{f}(\bx)$ and $\rK_{w_2}^{g}(\bu)=\rK_{w_2}^{g}(\bx)$. 
\end{proof}

\begin{remark}
Theorem \ref{tw:quasi_not} follows indirectly from \cite[Prop.~2.9]{gupta2023}, as the assumption $|f(0)|\wedge |g(0)|=\infty$ or $|f(1)|\wedge |g(1)|=\infty$ ensures that both 
$\rK_{w_1}^{f}$ and $\rK_{w_2}^{g}$ belong to the conjunctive or the disjunctive class.\footnote{An AF $\rA$ is said to belong to (i) \textit{conjunctive class} if $\rA(0,x)=0$ for all $x$, (ii) \textit{disjunctive class} if $\rA(x,1) = 1$ for all $x$ (cf.~\cite[Def.~2.8]{gupta2023}).}
\end{remark}

When the assumptions of Theorem~\ref{tw:quasi_not} fail to hold, further conditions are required to ensure that a~pair of quasi-linear means is admissible.

\begin{tw}\label{tw:quasi}
Let $\rK_{w_1}^{f}$ and $\rK_{w_2}^{g}$ 
be quasi-linear means with $|f(0)|\wedge|g(0)|<\infty$ and $|f(1)|\wedge|g(1)|<\infty$.
%%NIE MOZEMY ZAPISAĆ powyższego w następujący sposób:
%%%$\operatorname{Ran}(f),\operatorname{Ran}(g)\subseteq \mR$  ALBO 
%% $|f(0)|\wedge |g(0)| \wedge |f(1)|\wedge|g(1)|<\infty$.
Set $\widehat{f}=f|_{(0,1)}$, $\widehat{g}=g|_{(0,1)}$, and $\widehat{h}=\widehat{g}\circ \widehat{f}^{-1}$.
Then 
\begin{enumerate}[noitemsep, label = (\alph*)]
    \item 
    $(\rK_{w_1}^{f}, \rK_{w_2}^{g}) \in \adm$ with $w_1=w_2$ 
    if and only if $\widehat{h}$  is strictly convex or strictly  concave.
\end{enumerate}
%%%A~pair $(\rK_{w_1}^{f}, \rK_{w_2}^{g}) \in \adm$ with $\widehat{g}= a\widehat{f} +b$, $a,b\in\mR$ if and only if $a\neq 0$ and $w_1\neq w_2$.  
% USUWAMY bo z tego nie korzystamy dalej
Moreover, $(\rK_{w_1}^{f}, \rK_{w_2}^{g}) \in \adm$ with $w_1\neq w_2$ if one of the following conditions is satisfied: 
\begin{enumerate}[noitemsep, label = (\alph*), start = 2]
    \item when $\widehat{f}$ is strictly increasing and
    \begin{enumerate}[noitemsep, label = (b\arabic*)]
    \item 
    %the function 
    $\widehat{h}$ or $-\widehat{h}$ is convex and strictly increasing  if $w_1 < w_2$,  or

    \item 
    %the function 
    $\widehat{h}$ or $-\widehat{h}$ is convex  and strictly decreasing  if $w_1 > w_2$; 
\end{enumerate}

    \item when $\widehat{f}$ is strictly decreasing and
    \begin{enumerate}[noitemsep, label = (c\arabic*)]
    \item 
    %the function 
    $\widehat{h}$ or $-\widehat{h}$ is convex  and strictly decreasing  if $w_1 < w_2$,  or 

    \item 
    %the function 
    $\widehat{h}$ or $-\widehat{h}$ is convex  and strictly increasing if $w_1 > w_2$. 
\end{enumerate}
\end{enumerate}
\end{tw}
\begin{proof}
Put $\sJ_1 = \{\bz\in \sLi \mid 0<z_1\le z_2<1\}$ and $\sJ_2 = \{\bz\in \sLi \mid z_1=0\text{ or }z_2=1\}$. 
Clearly, $\sJ_1\cup \sJ_2 = \sLi$ and $\sJ_1 \cap \sJ_2 = \varnothing$.
To prove our assertions, we  need to verify
conditions~\ref{J1} and \ref{J2}, defined as follows:
\begin{enumerate}[noitemsep, label = (J\arabic*)]
    \item \label{J1}  for any $\bu,\bx \in \sJ_1$     
    \begin{align}\label{im:2}
        \Big(\rK_{w_1}^{f}(\bu) = \rK_{w_1}^{f}(\bx) \;\text{ and }\; \rK_{w_2}^{g}(\bu) = \rK_{w_2}^{g}(\bx)\Big) \quad \Rightarrow \quad \bu = \bx, 
    \end{align}

    \item \label{J2} implication \eqref{im:2} is true for any $\bu,\bx\in \sJ_2$.
\end{enumerate}

Obviously, $\widehat{h}\colon \operatorname{Ran}(\widehat{f}) \to \mR$ and $\operatorname{Ran}(\widehat{f})$ is an open interval  in $\mR$ due to continuity and strict monotonicity of $\widehat{f}$.
As $\widehat{f}$ and $\widehat{g}$ are both strictly monotone, to show \ref{J1}, we need to prove that  for any  $[s_1, s_2],[t_1, t_2] \in
\{\bz \in \sLi \mid  z_1,z_2\in\operatorname{Ran}(\widehat{f})  \}$, 
\begin{align}\label{im:3a}
\begin{cases}
    (1-v_1) s_1 + v_1s_2 = (1-v_1) t_1 + v_1t_2\\
      (1-v_2) \widehat{h}(s_1) + v_2\widehat{h}(s_2) = (1-v_2) \widehat{h}(t_1) + v_2\widehat{h}(t_2)   
\end{cases}\quad \Rightarrow\quad [s_1,s_2] = [t_1,t_2],  
\end{align} 
where $v_i=w_i$  if $\widehat{f}$ is strictly increasing, and $v_i= 1-w_i$ if $\widehat{f}$ is strictly decreasing, where $i\in \{1,2\}$.
It follows from  Theorem~\ref{tw:main} and  points (b)-(c) that condition \ref{adm2} holds, so implication \ref{J1} is established.
   
We now show \ref{J2}. Let $\bu,\bx \in \sJ_2$. 
Assume that
\begin{align}\label{eq:former}
\rK_{w_1}^{f}(\bu) = \rK_{w_1}^{f}(\bx)\quad  \text{and} \quad \rK_{w_2}^{g}(\bu) = \rK_{w_2}^{g}(\bx).
\end{align}

Due to  the definition of  $\sJ_2$, we 
need to examine few cases. 
\begin{enumerate}[noitemsep, label = (\alph*)]
   \item  Let $\bu= [0, u_2]$ and $\bx = [0, x_2]$ with $u_2,x_2 \in[0,1]$.
   As $|f(0)|\wedge |g(0)|<\infty$,  from  \eqref{eq:former},   we obtain $u_2=x_2$. 
 
    \item  Put $\bu = [u_1, 1]$ and $\bx = [x_1, 1]$ with $u_1,x_1\in[0,1]$.
    Due to $|f(1)| \wedge |g(1)|<\infty$, by  \eqref{eq:former}, we get $u_1=x_1$.

    \item  Set $\bu= [0, u_2]$ and $\bx = [x_1, 1]$ with $u_2,x_1\in [0,1]$. Consider two possibilities:
    \begin{enumerate}[noitemsep, label=(c\arabic*)]
        \item \label{C1}
        $|f(0)|\vee |f(1)| <\infty$. By~\cite[Thm.~4.17]{grabisch2009}, the function $\rK_{w_1}^{f}$ is strictly increasing on  $\sLi$, so the condition $\rK_{w_1}^{f}(0,u_2) = \rK_{w_1}^{f}(x_1,1)$ implies $x_1=0$ and $u_2=1$.
        In consequence, \eqref{eq:former} yields $\bu = \bx$. 
        The similar reasoning applies if $|g(0)|\vee |g(1)|<\infty$. 
        
        \item $|f(0)| \vee |g(1)|<\infty$. Observe that, if $|f(1)|<\infty$ or $|g(0)|<\infty$, then we get case \ref{C1}. 
        Thus, assume $|f(1)|\wedge |g(0)|=\infty$.
        Then condition~\eqref{eq:former}
        takes the form 
        \begin{align}
        \rK_{w_1}^{f}(0,u_2) = 1    \quad  \text{and} \quad 0= \rK_{w_2}^{g}(x_1,1).
        \end{align}
        Since the functions $u\mapsto \rK_{w_1}^{f}(0,u)$ and $x\mapsto \rK_{w_2}^{g}(x,1)$ are strictly increasing, we conclude that $u_2=1$ and $x_1=0$.
        Therefore, $u_1 = x_1$ and $u_2=x_2$.
        The case $|g(0)|\vee |f(1)|<\infty$ proceeds in a~similar manner. 
    \end{enumerate}
    The case $\bu = [u_1, 1]$ and $\bx = [0, x_2]$ with $u_1, x_2\in[0,1]$ follows by a~similar argument, so we omit it.
\end{enumerate}
\end{proof}

\begin{corollary}\label{cor:full}
Let $\rK_{w_1}^{f}$ and $\rK_{w_2}^{g}$ 
be quasi-linear means. 
Set $\widehat{f}=f|_{(0,1)}$, $\widehat{g}=g|_{(0,1)}$, and $\widehat{h}=\widehat{g}\circ \widehat{f}^{-1}$.
Then 
$(\rK_{w_1}^{f}, \rK_{w_2}^{g}) \in \adm$ with $w_1=w_2$ if and only if 
$|f(0)|\wedge|g(0)|<\infty$, $|f(1)|\wedge |g(1)|<\infty$, 
and $\widehat{h}$  is strictly convex or strictly  concave.  
\end{corollary}
\begin{proof}
The implication ``$\Leftarrow$'' follows from Theorem~\ref{tw:quasi}, (a), while the converse implication is a~consequence of Theorems~\ref{tw:quasi_not} and~\ref{tw:quasi}\,(a) 
\end{proof}

Due to the complexity of conditions (b)–(c) in Theorem~\ref{tw:quasi}, 
we include Table~\ref{tab:1} to help make their verification easier.
\begin{table}[h!]
\centering
\begin{tabular}{ |c|c|c| }
\hline
$w_1$ vs  $w_2$ 
&  $f|_{(0,1)}$ 
& $g|_{(0,1)} \circ (f|_{(0,1)})^{-1}$ \\
\hline
\multirow{4}{*}{$w_1 < w_2$} 
 & \multirow{2}{*}{str. increasing}  & convex and str. increasing \\ \cline{3-3}
 &  & concave and str. decreasing \\ \cline{2-3}
 & \multirow{2}{*}{str. decreasing}  & convex and str. decreasing \\\cline{3-3}
 &  & concave and str. increasing \\
\hline
\multirow{4}{*}{$w_1 > w_2$} 
 & \multirow{2}{*}{str. increasing} & convex and str. decreasing \\\cline{3-3}
 &  & concave and str. increasing \\\cline{2-3}
 & \multirow{2}{*}{str. decreasing} & convex and str. increasing \\\cline{3-3}
 &  & concave and str. decreasing \\
\hline
\end{tabular}
\caption{
Each row represents a~different configuration of assumptions on  $f|_{(0,1)}$ and $g|_{(0,1)}\circ (f|_{(0,1)})^{-1}$, which, by Theorem~\ref{tw:quasi}\,(b)-(c), determine the admissibility of 
$(\rK^f_{w_1}, \rK^{g}_{w_2})$ with a~fixed $w_1\neq w_2$ if  $|f(0)|\wedge |g(0)|<\infty$ or $|f(1)|\wedge |g(1)|<\infty$. 
Abbreviation ``str.'' stands for ``strictly''.}
\label{tab:1}
\end{table}

%\paragraph{\textbf{Particular cases}}
We provide several special cases derived from the obtained results  (Theorems \ref{tw:quasi_not} and \ref{tw:quasi}, and Corollary~\ref{cor:full})  
for some quasi-linear means with fixed weights.

\begin{example}\label{ex:wrmp}  
A~weighted root-mean-power generated by the function
$f(x) = x^{\gamma}$ with $\gamma \in \mathbb{R} \setminus \{0\}$ is defined as
\begin{align}\label{QL1}
    \mathbf{WRM}_w^{\gamma}(\bu) = ((1-w)u_1^{\gamma} + w u_2^{\gamma})^{1/\gamma}.\tag{QL1}
\end{align}
Theorem \ref{tw:quasi_not} yields 
$(\mathbf{WRM}_{w_1}^{\alpha}, \mathbf{WRM}_{w_2}^{\beta})\not\in \adm$
for any $w_1,w_2$ and $\alpha,\beta<0$.
%\sout{both $\alpha,\beta$ are negative}. 
Next, from Corollary \ref{cor:full}, 
we conclude that  $(\mathbf{WRM}_{w_1}^{\alpha}, \mathbf{WRM}_{w_2}^{\beta}) \in \adm$ with $w_1 = w_2$ for any 
$\alpha\neq \beta$ such that at least one of them is positive.
Using Table~\ref{tab:1}, we get $(\mathbf{WRM}_{w_1}^{\alpha}, \mathbf{WRM}_{w_2}^{\beta}) \in \adm$ with:
\begin{enumerate}[noitemsep, label=(\alph*)]
    \item 
     $w_1<w_2$ if $\alpha< 0 <\beta$ or $0 < \alpha \le \beta$; 
    
    \item 
    $w_1 > w_2$ if $\beta < 0 < \alpha$ or $0 < \beta \le \alpha$. 
\end{enumerate}
\end{example}

\begin{example}\label{ex:wem}
Consider a~pair of weighted exponential means. 
Recall that a~weighted exponential mean is a~quasi-linear mean generated by the function $f(x) = e^{\gamma x}$ with $\gamma \in \mathbb{R} \setminus \{0\}$ and has the form
\begin{align}\label{QL2}
    \mathbf{WEM}_w^{\gamma}(\bu) = \gamma^{-1}\log((1-w)e^{\gamma u_1}+we^{\gamma u_2}). \tag{QL2}
\end{align}
Since the generating function takes only finite values, 
the admissibility of  a~pair of weighted exponential means follows from Corollary~\ref{cor:full},
which states that $(\mathbf{WEM}_{w_1}^{\alpha}, \mathbf{WEM}_{w_2}^{\beta}) \in \adm$ with $w_1 = w_2$
for any distinct $\alpha,\beta$. 
%{\mk such that $\alpha\, \beta\neq 0?$} {\mb przyjąłem wszędzie styl, że jak pisze "for any $\alpha,\beta$" to biore wszystkie dopuszczalne parametry, czyli bez zer, bo tak jest w definicji ->  taki styl jest wszedzie. Oczywiście możemy zmienić aby podawać od razu wszystkie wartości parametrów, ale to musimy to ustalić aby wszędzie tak pisać... W niektórych miejscach mogą pojawić sie problemy w redagowaniu treści dlateog zdecydowałem sie na wersję krótszą}.
Furthermore, by Table~\ref{tab:1}, $(\mathbf{WEM}_{w_1}^{\alpha}, \mathbf{WEM}_{w_2}^{\beta}) \in \adm$ with $w_1 < w_2$ if 
$\alpha \le \beta$, and  with $w_1 > w_2$ if  $\alpha \ge \beta$.
\end{example}

\begin{example}\label{ex:wgm}  We take a~weighted geometric mean generated by   $f(x) = \log x$, 
\begin{align}\label{QL3}
    \mathbf{WGM}_w(\bu) = u_1^{w}u_2^{1-w} \tag{QL3}  
\end{align}
and  a~quasi-linear mean generated by  $f(x)=\log (x(1-x)^{-1}),$ 
\begin{align}\label{QL4}
    \mathbf{WM}_w(\bu)
    =\frac{u_1^{1-w}u_2^{w}}{u_1^{1-w}u_2^{w}+(1-u_1)^{1-w}(1-u_2)^{w}}\tag{QL4}
\end{align}
 under the convention that $0/0 = 0$ (see~\cite[Ex.~4.7]{grabisch2009}).
Theorem~\ref{tw:quasi_not} gives 
$(\mathbf{WGM}_{w_1},\mathbf{WGM}_{w_2}) \notin \adm$ and $(\mathbf{WM}_{w_1}, \mathbf{WM}_{w_2}) \notin\adm$ for any $w_1,w_2$. 
\end{example}

We now turn to the analysis of admissibility for pairs of AFs derived from combining cases \eqref{QL1}--\eqref{QL4}.

\begin{example}\label{ex:3.6}
From Theorem~\ref{tw:quasi_not}, Corollary~\ref{cor:full}, and Table~\ref{tab:1}, it follows that:
\begin{enumerate}[noitemsep, label=(\alph*)]
    \item $(\mathbf{WGM}_{w_1}, \mathbf{WEM}_{w_2}^{\alpha})\in\adm$ with: 
    \begin{enumerate}[noitemsep, label=(a\arabic*)]
        \item $w_1 = w_2$ if and only if $\alpha \in [-1,\infty)\setminus\{0\}$;
        \item $w_1 < w_2$ if $\alpha\in [-1,\infty)\setminus\{0\}$; 
    \end{enumerate}
    %%TEN sam wynik wychodzi dla $( \mathbf{WEM}_{w_1}^{\alpha}, \mathbf{WGM}_{w_2}) \in\adm$}. Kuba sprawdził.

    \item $(\mathbf{WM}_{w_1}, \mathbf{WEM}_{w_2}^{\alpha})\notin \adm$ with $w_1 = w_2$ for any $\alpha$, as  $\widehat{h}(x)=\exp(\alpha e^x/(1+e^x))$, $x\in \mR$, is neither strictly convex nor strictly concave for any $\alpha \neq 0$; 

    \item  $(\mathbf{WM}_{w_1}, \mathbf{WRM}_{w_2}^{\alpha})\not\in\adm$ with $w_1 = w_2$ for any $\alpha$,  or with $w_1\neq w_2$ for $\alpha <0$, as  $\widehat{h}(x) = (e^{x} / (1+ e^{x}))^{\alpha}$, $x\in \mR$, is neither (strictly) convex  nor (strictly) concave for any $\alpha \neq 0$; 
   
    \item $(\mathbf{WM}_{w_1}, \mathbf{WGM}_{w_2})
    \notin \adm$ for any $w_1,w_2$; 

    \item  $(\mathbf{WRM}_{w_1}^{\alpha}, \mathbf{WGM}_{w_2})\in\adm$ with $w_1 \ge w_2$ for $\alpha>0$, and $(\mathbf{WRM}_{w_1}^{\alpha}, \mathbf{WGM}_{w_2})\not\in\adm$ for any $w_1, w_2$ and $\alpha<0$; 

    \item 
     $(\mathbf{WRM}_{w_1}^{\alpha}, \mathbf{WEM}_{w_2}^{\beta}) \in \adm$ with: 
    \begin{enumerate}[noitemsep, label=(f\arabic*)]
        \item    $w_1 = w_2$ if and only if
        $(\alpha,\beta) \in \{(a,b)  \in (\mR\setminus \{0\})^2 \mid a< 1 \text{ and } b> a-1\} \cup \{(a,b)\in (\mR\setminus\{0\})^2\mid a> 1 \text{ and } b < a-1\}$; 

        \item $w_1 < w_2$ if 
         $(\alpha,\beta)\in \{(a,b)\in (\mR\setminus\{0\})^2\mid a\le 1 \text{ and } b\ge a-1\}$; 
         
        \item $w_1 > w_2$ if $(\alpha,\beta)\in \{(a,b)\in (\mR\setminus\{0\})^2\mid a\ge 1 \text{ and } b\le a-1\}$.
    \end{enumerate}
\end{enumerate}
\end{example}

%\begin{remark}\label{rem:1}
%When the compositions $\widehat{g}\circ \widehat{f}^{-1}$ are neither convex nor concave, we are unable to determine the admissibility of their quasi-linear means when $w_1 \neq w_2$ and the function $h$ is non-affine. Such situations arise for the pairs $(\mathbf{WM}_{w_1}, \mathbf{WEM}_{w_2}^{\alpha})$ for any $\alpha$ and $(\mathbf{WM}_{w_1}, \mathbf{WRM}_{w_2}^a)$ for $\alpha >0$.\end{remark}

%%%%
%Although Proposition~\ref{pro:3.5}  shows that 
\medskip

We conclude this section with the observation that the converse implications in Theorem~\ref{tw:quasi}\,(b)-(c) do not hold (see Section~\ref{sec:aux}). However, combining Proposition~\ref{pro:3.5} with the argument from the proof of Theorem~\ref{tw:quasi} yields the following result.

\begin{tw}\label{tw:quasi2}
Let $\rK_{w_1}^{f}$ and $\rK_{w_2}^{g}$ 
be quasi-linear means with $|f(0)|\wedge|g(0)|<\infty$ and $|f(1)|\wedge|g(1)|<\infty$.
Set $\widehat{f}=f|_{(0,1)}$, $\widehat{g}=g|_{(0,1)}$, and $\widehat{h} = \widehat{g}\circ \widehat{f}^{-1}$. 
\begin{enumerate}[noitemsep, label = (\alph*)]
    \item Suppose that  both $\widehat{f}$ and $\widehat{h}$ are strictly increasing or strictly decreasing. Then  
    $(\rK_{w_1}^{f}, \rK_{w_2}^{g}) \in \adm$ for any $w_1 < w_2$ (resp., $w_1 > w_2$) if and only if $\widehat{h}$ is convex (resp., concave).

    \item 
    Assume that  $\widehat{f}$ and $\widehat{h}$ are strictly monotone with the opposite monotonicity. Then  $(\rK_{w_1}^{f}, \rK_{w_2}^{g}) \in \adm$ for any $w_1 < w_2$ (resp., $w_1 > w_2$) if and only if $\widehat{h}$ is concave (resp., convex).
\end{enumerate}
\end{tw}

%{\mbTwierdzenie mówi nam, że bez względu na wartość wag o określonej strukturze  $w_1< w_2$ or $w_1 > w_2$, funkcja $h$ musi być odpowiedniego kształtu. Theorem~\ref{tw:quasi2} can be utilized in neural networks to design aggregation operators that preserve admissible order over interval-valued inputs, which is particularly useful in models handling uncertainty or imprecise data.}

\subsection{\textbf{Archimedean $t$-norms}}\label{sec:arch}

A~$t$-norm
%\footnote{A~binary function $\rT\colon [0,1]^2\to[0,1]$ is said to be a~\textit{$t$-norm} if it is (i) commutative ($\rT(a,b) = \rT(b,a)$ for any $a,b$), (ii) associative ($\rT(\rT(a,b), c) = \rT(a, \rT(b, c))$ for any $a,b,c$), (iii) increasing ($\rT(a,c)\le \rT(b,d)$ whenever $a\le b$ and $c\le d$), and (iv) $\rT(a,1)=a$ for all $a$ \cite{klement2000, nguyen2006}.} 
$\rT$ is called \textit{Archimedean} if it is continuous and $\rT(x,x)<x$ for any $x\in (0,1)$~\cite{klement2000, nguyen2006}.
When $\rT$ is an Archimedean $t$-norm, its restriction $\rT|_{\sLi}$  is an AF.
%From~\cite{klement2000, nguyen2006} it follows that 
Archimedean $t$-norms are characterized by their additive generator $\rt\colon [0,1] \to [0, \infty]$ via
$\rT(a,b) =\rt^{-1}( (\rt(a)+\rt(b))\wedge \rt(0))$
%\rt^{-1}(\min\{\rt(a) + \rt(b), \rt(0)\})$ 
for all $a,b$, where $\rt$ is a~strictly decreasing and continuous function such that $\rt(1)=0$. 
Depending on $\rt(0)$, $\rT$ is  strict ($\rt(0)=\infty$) or nilpotent ($\rt(0)<\infty$)~\cite[Prop.~3.29]{klement2000}.
%Any Archimedean $t$-norm can be represented by its additive generator $\rt\colon [0,1] \to [0,\infty]$ in the following way $\rT(a,b) = \rt^{-1}\big(\min\{\rt(a) + \rt(b), \rt(0)\}\big)$ for any $a,b\in [0,1]$.  The additive generator $\rt$ of $\rT$ is a~strictly decreasing and continuous function satisfying $\rt(1) = 0$. Furthermore, each Archimedean $t$-norm is either strict ($\rt(0)=\infty$) or nilpotent ($\rt(0)<\infty$)~\cite[Prop.~3.29]{klement2000}. 
%If $\rt(0) = \infty$, then the $t$-norm simplifies to $\rT(a,b) = \rt^{-1}(\rt(a) + \rt(b))$ for any $a,b$.  

A~$t$-conorm
%\footnote{A~binary function $\rS\colon [0,1]^2\to[0,1]$ is said to be a~\textit{$t$-conorm} if it is (i) commutative, (ii) associative, (iii) increasing, and (iv) $\rS(a,1)=1$ for all $a$ \cite{klement2000, nguyen2006}.} 
$\rS$ is called \textit{Archimedean} if it is continuous and $\rS(x,x) > x$ for any $x\in (0,1)$~\cite{nguyen2006}.
Any Archimedean $t$-conorm can also be represented by its additive generator $\rs$, as
$\rS(a,b)=\rs^{-1}( (\rs(a)+\rs(b))\wedge \rs(1))$
%\rs^{-1}( \min\{ \rs(a)+\rs(b), \rs(1)\}) $ 
for any $a,b$, where $\rs\colon [0,1] \to [0,\infty]$ is a~strictly increasing and continuous function  such that $\rs(0) = 0$.
If $\rs(1) = \infty$, then $\rS$ is a~strict Archimedean $t$-conorm while it is niplotent for $\rs(1) < \infty$.

In~\cite[Prop.~2.9]{gupta2023}, it was shown that $(\rT|_{\sLi}, \widehat{\rT}|_{\sLi}) \notin \adm$ and $(\rS|_{\sLi}, \widehat{\rS}|_{\sLi})\notin \adm$, where $\rT,\widehat{\rT}$ are Archimedean $t$-norms, and $\rS, \widehat{\rS}$ are Archimedean $t$-conorms.
Therefore, we will focus on examining the admissibility of pairs of the form $(\rT|_{\sLi}, \rS|_{\sLi})$.

\begin{tw}\label{tw:norm}
Let $\rT$ be a~strict Archimedean $t$-norm with the additive generator $\rt$ and $\rS$ be a~strict Archimedean $t$-conorm with the additive generator $\rs$. 
Set $\ot=\rt|_{(0,1)}$ and $\os=\rs|_{(0,1)}$. 
%If $\ot \circ \os^{-1}$ is a~strictly convex/concave function, then $(\rT|_{\sLi}, \rS|_{\sLi})\in \adm$.
Then, $(\rT|_{\sLi}, \rS|_{\sLi})\in \adm$ if and only if 
$\os \circ \ot^{-1}$ is strictly convex or strictly concave.
\end{tw}
\begin{proof}
%{\mb [może tak bedzie czytelniej:] 
Verifying the admissibility of the pair $(\rT|_{\sLi}, \rS|_{\sLi})$ is equivalent to determining the admissibility of the pair $(\rK_{0.5}^{\rt}, \rK_{0.5}^{\rs})$, since $\operatorname{Ran}(t) = \operatorname{Ran}(s) = [0, \infty]$. Therefore, the statement follows from Corollary~\ref{cor:full} with $w_1=w_2=0.5$, 
$f = \rt$, and $g=\rs$, as $\rt(0) \wedge \rs(0) = 0 = \rt(1) \wedge \rs(1)$.
%%STARA WERSJA
%Note that condition \ref{Adm} with $\rA = \rT|_{\sLi}$ and $\rB=\rS|_{\sLi}$ is equivalent to:   $\rK_{0.5}^{\rt}(\bu) = \rK_{0.5}^{\rt}(\bx)$ and $\rK_{0.5}^{\rs}(\bu) = \rK_{0.5}^{\rs}(\bx)$ can only hold if $\bu = \bx$, for any $\bu,\bx \in \sLi$, {\mb as $\mathrm{Ran}(t) =\mathrm{Ran}(s) = [0,\infty]$.} So, the statement follows from Corollary~\ref{cor:full} with $w_1=w_2=0.5$, $f = \rt$, and $g=\rs$, as $\rt(1) \wedge \rs(1) = 0=\rt(0) \wedge \rs(0)$.
\end{proof}

\begin{tw}\label{tw:norm_not}
$(\rT|_{\sLi}, \rS|_{\sLi})\notin \adm$ for any nilpotent Archimedean $t$-norm $\rT$ or any  nilpotent Archimedean $t$-conorm $\rS$.
\end{tw}
\begin{proof}
It is enough to indicate two distinct intervals $\bu, \bx \in \sLi$ for which the following system of equations holds: 
\begin{align}\label{arch:2}
\begin{cases} 
\rt^{-1}\big((\rt(u_1)+\rt(u_2)) \wedge \rt(0)\big) = \rt^{-1} \big((\rt(x_1)+\rt(x_2))\wedge \rt(0)\big),\\ \rs^{-1}\big((\rs(u_1)+\rs(u_2)) \wedge\rs(1)\big) = \rs^{-1}\big((\rs(x_1)+\rs(x_2))\wedge \rs(1) \big). 
\end{cases} 
\end{align}
Firstly, assume that $\rT$ is a~nilpotent Archimedean $t$-norm and $\rS$ is a~nilpotent Archimedean $t$-conorm. 
Put $D_{\rt} =\{a\in [0,1] \mid \rt(a) \ge 0.5\rt(0)\}$  and $D_{\rs} = \{a\in [0,1] \mid \rs(a) \le 0.5 \rs(1)\}$,
where $\rt$ and $\rs$ are the additive generators of $\rT$ and $\rS$, respectively.
Since $\rt$ is a~continuous and strictly decreasing function with $\rt(0)<\infty$, there exists $m_{\rt} \in (0,1)$ such that $D_{\rt} = [0, m_{\rt}]$.
%tutaj też domyślnie korzystamy z intermediate value theorem
Similarly, there exists $m_{\rs} \in (0,1)$ such that $D_{\rs} = [0, m_{\rs}]$. 
Set $m = m_{\rs} \wedge m_{\rt}>0$.
Note that 
\begin{align}\label{arch:0}
 \rt(z_1) + \rt(z_2) \ge \rt(0)\quad \text{and} \quad  \rs(z_1) + \rs(z_2) \le \rs(1) 
\end{align} 
for any $z_1,z_2 \in [0, m]$.
%$\rt(z_1) + \rt(z_2) \ge \rt(0)$ and $\rs(z_1) + \rs(z_2) \le \rs(1)$.
Due to~\eqref{arch:0}, it is enough to find two distinct $\bu, \bx \in \sLi$ fulfilling
\begin{align}\label{arch:1}
   u_2\vee x_2\le m\qquad \text{and} \qquad  \rs(u_1) + \rs(u_2) = \rs(x_1) + \rs(x_2),
\end{align}
as \eqref{arch:1} implies~\eqref{arch:2}.
Let $u = 0.5 m$,  $d_1 = \rs(u) - \rs(0) >0$, and $d_2 = \rs(m) - \rs(u)>0$. 
We  consider three cases.
\begin{enumerate}[noitemsep, label=(\alph*)]
    \item Let $d_1 < d_2$.
    Clearly, $\rs(u) < \rs(u) + d_1 < \rs(u) + d_2 = \rs(m)$. By the intermediate value theorem there exists $x_2\in (u,m)$ such that $\rs(x_2) = \rs(u)+ d_1$. Thus, $u\vee x_2 < m$ and $\rs(u) + \rs(u) = \rs(u) - d_1 + \rs(u) + d_1 = \rs(0) + \rs(x_2)$.
    In consequence, 
    %as $u\vee x_2 < m$,
    condition \eqref{arch:1} with $\bu = [u,u]$ and $\bx = [0, x_2]$ is satisfied.
     
    \item Set $d_1 > d_2$. Hence $\rs(0) = \rs(u) - d_1 < \rs(u)- d_2 < \rs(u)$, the intermediate value theorem yields that there exists $x_1\in (0,u)$ such that $\rs(x_1) = \rs(u)- d_2$. 
    Hence, $\rs(u) + \rs(u) = \rs(u) - d_2 + \rs(u) + d_2 = \rs(x_1) + \rs(m)$. Therefore, the distinct intervals that satisfy \eqref{arch:1} are of the form $\bu = [u,u]$ and $\bx = [x_1, m]$.

    \item If $d_1 = d_2$, then intervals $\bu = [u,u]$ and $\bx = [0, m]$ meet condition \eqref{arch:1}.
\end{enumerate}
Assume now that $\rT$ is a~strict Archimedean $t$-norm and $\rS$ is a~nilpotent Archimedean $t$-conorm. Consider the set $D_{\rs} = \{a\in [0,1] \mid \rs(a) \ge 0.5\rs(1)\}$. 
%\sout{Assume now, that $\rT$ is a~strict Archimedean $t$-norm and $\rS$ is a~nilpotent Archimedean $t$-conorm, consider the set $D_{\rs} = \{a\in [0,1] \mid \rs(a) \ge 0.5\rs(1)\}$.} 
Clearly, 
%Given the properties of the additive generator
%$\rs$, it follows that 
$D_{\rs} = [l_{\rs}, 1]$ for some $l_{\rs} \in (0,1)$.
For any $z_1,z_2 \in [l_s, 1]$, we have 
$\rt(z_1) + \rt(z_2) < \rt(0) = \infty$ and $\rs(z_1) + \rs(z_2) \ge \rs(1)$. 
Consequently, it is enough to find two distinct intervals $\bu, \bx \in \sLi$ satisfying
\begin{align}
   u_1\wedge x_1 \ge l_{\rs}\quad  \text{and} \quad \rt(u_1) + \rt(u_2) = \rt(x_1) + \rt(x_2). 
\end{align}
Put
$u=0.5 (l_{\rs}+1)$, $d_1 = \rt(u)-\rt(1)>0$, and $d_2 = \rt(l_{\rs})-\rt(u) > 0$. 
Following a~similar approach as in the first case discussed, one can construct such intervals.
When $\rT$ is a~nilpotent Archimedean $t$-norm and $\rS$ is a~strict Archimedean $t$-conorm, the proof is similar, so we omit it.
The proof is complete.
\end{proof}

Using Theorems \ref{tw:norm} and \ref{tw:norm_not},  we are able to  determine whether the pair $(\rT|_{\sLi}, \rS|_{\sLi})$ is admissible, where $\rT$ and $\rS$ are an Archimedean $t$-norm and an Archimedean $t$-conorm, respectively. This extends the results of~\cite[Rem.~2.11\,(ii)]{gupta2023}.

Since an Archimedean copula is a~special case of an Archimedean $t$-norm \cite[Thm.~2.2.9]{alsina2006}, all the above results can be applied to the pair $(\rC|_{\sLi}, \rC^{\ast}|_{\sLi})$, where $\rC$ is an Archimedean copula, and $\rC^{\ast}$ is an Archimedean co-copula, defined as $\rC^{\ast}(a,b) = 1 -\widehat{\rC}(1-a, 1-b)$ for some copula $\widehat{\rC}$.

%{\mb [pytanie czy piszemy wyniki dla $(K^f_{w}, \rT)$, gdzie $\rT$ jest archimedowska $t$-(co)norma? Wg mnie, nie; zostawmy to pytanie recenzentowi. Z drugiej strony czy mieszanie średnich z $t$-normami jest dobre? czyli mieszanie zwykłego świata z fuzzy?]} 

\subsection{\textbf{Strictly Schur-convex functions}}\label{sec:schur}

Recall that a~function $\rF\colon \sLi\to \mR$ 
is called \textit{Schur-convex} (resp., \textit{Schur-concave}) if the inequality 
\begin{align}\label{sn:1}
     \rF(\bu) \le \rF(\bx)\quad (\text{resp., } \rF(\bu)\ge \rF(\bx)),
\end{align}
holds for all $\bu,\bx\in \sLi$ such that $u_1+u_2 = x_1 + x_2$ and $u_2\le x_2.$ 
%% \rF(x,y) = x nie jest strictly Schur conve
%% Przykładem Schur-convex function is: $\vee$.... Schur-concave: $\wedge$ -> patrz Wiki
The function $\rF$ is called \textit{strictly Schur-convex} (resp., \textit{strictly Schur-concave}) if inequality \eqref{sn:1} is strict whenever 
$u_1+u_2 = x_1 + x_2$ and $u_2<x_2$.
For example, the AF $\rA$ defined by $\rA(u_1, u_2) = 0.5(f(u_1) + f(u_2))$  with a~strictly convex  (resp., strictly concave) and increasing function $f\colon [0,1] \to [0,1]$ such that $f(0) = 1-f(1) =0$, 
is strictly Schur-convex (resp., strictly Schur-concave) \cite[p.\,92]{marshall2011}.
See also \cite[Sec. 4]{alsina2006} and \cite[Sec.~3]{marshall2011} for more details on Schur-convexity. 

\begin{tw}\label{tw:schur}
Let $\rA(u_1, u_2) = 0.5(f(u_1) + f(u_2))$ and $\rB(u_1, u_2) = 0.5(g(u_1) + g(u_2))$ for any $u_1, u_2 \in [0,1]$,  where both $f$ and $g$ are increasing bijections on $[0,1]$. 
Then $(\rA, \rB) \in \adm$ if and only if $g\circ f^{-1}$ is a~strictly convex or strictly concave function.
\end{tw}
\begin{proof}
Note that  condition~\ref{Adm} for  the pair $(\rA, \rB)$ is equivalent to condition~\ref{adm2} with $\sI = [0,1]$, $h = g \circ f^{-1}$, and $v_1 = v_2 = 0.5$. Therefore, the statement follows from Theorem~\ref{tw:main}\,(a).
%As $f$ is a~strictly increasing function, {\mb [czy to założenie ma jakiś wpływ gdy wagi są takie same?]} so the statement follows from Theorem~\ref{tw:main}\,(a) with $\sI = [0,1]$, $h= g\circ f^{-1}$, and $v_1 = v_2 = 0.5$.
\end{proof}

\section{\textbf{Coincidence 
between $\leqa_{\rA,\rB}$ and $\leqa_{(\alpha, \beta)}$}}\label{sec:coin}

Let $\leqa$ and $\unlhd$ be binary relations on $\sLi$. 
We say that $\leqa$ \textit{coincides with}  $\unlhd$ 
whenever $\leqa\,= \unlhd$, that is, for all $\bu, \bx \in \sLi$,  $\bu \leqa \bx$ if and only if  $\bu \unlhd \bx$. 
%for all $\bu,\bx\in \sLi$ we have $\bu \leqa \bx$ if and only if $\bu \unlhd \bx$. \sout{In this case, we write $\leqa\,= \unlhd$.}
For instance,  if $\rB_1$ and $\rB_2$ are AFs such that the functions $[0,x_2]\ni x\mapsto \rB_1(x,x_2)$ and $[x_1, 1]\ni x\mapsto \rB_2(x_1,x)$ are strictly increasing for any fixed $x_1,x_2$, then  
$\leqa_{(1,0)} \,=\,\leqa_{\rK_1, \rB_1}$ and
$\leqa_{(0,1)}\,=\,\leqa_{\rK_0, \rB_2}$, as $(\rK_1, \rB_1), (\rK_0, \rB_2)\in \adm$ (see Example~\ref{ex:2.6}).
We now identify 
classes of AFs $\rB$ for which the coincidence between $\leqa_{\rK_{0.5}, \rB}$ and the $(0.5, \beta)$-order holds.

\begin{proposition}\label{pro:07}
Let  $(\rK_{0.5},\rB) \in \adm$. 
If $\leqa_{\rK_{0.5},\rB}$ coincides with $(0.5, 1)$-order (resp., $(0.5, 0)$-order), then  $\rB$ is Schur-convex  (resp., Schur-concave). 
Moreover, if $\rB$ is a~strictly Schur-convex (resp., strictly Schur-concave), then $\leqa_{\rK_{0.5}, \rB}$ coincides with $(0.5, 1)$-order (resp., $(0.5, 0)$-order).
\end{proposition}
\begin{proof}  
The admissible order $\leqa_{\rK_{0.5},\rB}$ coincides with $(0.5, 1)$-order if and only if the statements $p$ and $q$ 
have the same logical value, where 
$p=\,$``$u_1+u_2=x_1+x_2$ and $u_2\le x_2$'' and  $q=\,$``$u_1+u_2=x_1+x_2$ and  $\rB(u_1,u_2)\le \rB(x_1,x_2)$''. 
 From the fact that $p$ implies $q$, it follows that $\rB$ is Schur-convex. 
 Further, if $\rB$ is strictly Schur-convex, then $q$ implies $p$.  
In fact, suppose that $u_2>x_2$. Then, by strict Schur-convexity, 
$\rB(x_1,x_2)<\rB(u_1,u_2),$ a~contradiction with $\rB(u_1,u_2)\le \rB(x_1,x_2)$. 
The proof for the $(0.5, 0)$-order is similar.%tutaj druga i koniec
%%[komentarz:
%%$$p\Rightarrow q \quad\Rightarrow\quad \rB \text{ is Schur-convex, }$$
%%powyższa implikacja jest nam potrzebna, bo jest pierwszą implikacją w tezie twierdzenia, natomiast implikacja
%%$$\rB \text{ is Schur-convex } \quad\Rightarrow\quad p\Rightarrow q,$$
%%jest również potrzebna do dalszej części dowodu, bo teraz będziemy pokazywać, że
%%$$\rB \text{ is strictly Schur-convex } \quad\Rightarrow\quad q\Rightarrow p,$$
%%co da nam drugą implikację z tezy.]
\end{proof}

Recall that $\leqa_{(\alpha,\beta)} \,= \,\leqa_{(\alpha, 1)}$ for $\alpha<\beta$ 
and $\leqa_{(\alpha,\beta)}\,=\,\leqa_{(\alpha,0)}$ for $\alpha>\beta$ (see \cite{bustince2013}), so from  Proposition~\ref{pro:07}, 
it follows that if $\leqa_{\rK_{0.5}, \rB}$ coincides with $(0.5, \beta)$-order, then $\rB$ is Schur-convex or Schur-concave function. 

%Teraz przypadki na nie
Next, we consider the cases when  $\leqa_{\rA,\rB}\, \neq\, \leqa_{(\alpha,\beta)}$.  
We begin by presenting an example of  admissible order generated by a pair $(\rA,\rB) \in \adm$ that does not coincide with any $(\alpha, \beta)$-order.

\begin{example}\label{ex:5} 
Let $\rA(\bz) = 0.5 (z_1^2 + z_2^2)$ and $\rB(\bz) = 0.5(\sqrt{z_1} + \sqrt{z_2})$ for any $\bz\in \sLi$. 
By Theorem~\ref{tw:schur}~with $f(x) = x^2$ and $g(x) = \sqrt{x}$, $(\rA,\rB)\in\adm$.
For $\bu =  [0.36, 0.82]$ and $\bx = [0.08, 0.92]$, we get $0.401 = \rA(\bu) < \rA(\bx) = 0.4264$ and
\begin{align*}
    \rK_{\alpha}(\bu) > \rK_{\alpha}(\bx) \quad \Leftrightarrow\quad \alpha < \frac{14}{19} < 0.737.
\end{align*}
Thus $\bu \prec_{\rA,\rB} \bx$ and $\bx \prec_{(\alpha, \beta)} \bu$ for $\alpha < 14/19$.
For $\bu =  [0.27, 0.71] $ and $\bx =[0.57, 0.59]$, we obtain $0.2885  = \rA(\bu)  < \rA(\bx) = 0.3365$ and
\begin{align*}
    \rK_{\alpha}(\bu)  > \rK_{\alpha}(\bx) \quad \Leftrightarrow\quad \alpha > \frac{5}{7} > 0.714.
\end{align*}
Hence, $\bu \prec_{\rA,\rB} \bx$ and $\bx \prec_{(\alpha, \beta)} \bu$ for $\alpha > 5/7$.
In consequence, the relation $\leqa_{\rA,\rB}$ does not coincide with the $(\alpha, \beta)$-order for any $\alpha,\beta$.
\end{example}

%\begin{lemma}\cite[Sec.~1.4]{niculescu2018}\label{lem:9.1}
%A~continuous function $f\colon [0, 1] \to \mR$ is strictly convex (resp. strictly concave) if and only if  $f(x+\delta)-f(x)< f(y+\delta)-f(y)$ (resp. $f(x+\delta)-f(x) > f(y+\delta)-f(y)$) for any $x< y$ and $\delta>0$ such that $x+\delta, y+\delta \in [0,1]$. \end{lemma}

We now present some  classes of AFs for which the relation generated by an admissible pair of AFs does not coincide with a~given $(\alpha, \beta)$-order.

\begin{proposition}\label{pro:5}
If $\alpha \le 0.5$ (resp., $\alpha\ge 0.5$) and $\beta\neq \alpha$, then  there exists a pair $(\rA,\rB)\in \adm$  with strictly Schur-convex (resp., strictly Schur-concave) function $\rA$  for which an admissible order $\leqa_{\rA,\rB}$ does not coincide with $(\alpha, \beta)$-order. 
\end{proposition}
\begin{proof} 
We  take $\alpha\le 0.5,$ the proof of the case $\alpha\ge 0.5$ is analogous. 
Let $(\rA, \rB)\in\adm$ be the pair of AFs defined in Theorem~\ref{tw:schur}. 
Moreover, we assume that $f$ is a~strictly convex function on $[0,1]$, so $\rA$ is a~strictly Schur-convex function (see Section~\ref{sec:schur}).
Let $\bu,\bx \in \sLi$ be such that $u_1 < u_2 < x_2$  and $\rK_{\alpha}(\bu) = \rK_{\alpha}(\bx)$. 
Then, $\bu \subset \bx$ (see \cite[Lem.~1]{sussner2024}).
We  show that 
\begin{align}\label{coi:n1}
    f(u_1) + f(u_2) <  f(x_1) + f(x_2).
\end{align}
For $\alpha =0$, it is clear that   inequality \eqref{coi:n1} is true, as $x_1 = u_1$.
Assume that $\alpha > 0$. 
Since $f$ is strictly convex,\footnote{
A~continuous function $f\colon [0, 1] \to \mR$
is strictly convex (resp., strictly concave) if and only if  $f(x+\delta)-f(x)< f(y+\delta)-f(y)$ (resp., $f(x+\delta)-f(x) > f(y+\delta)-f(y)$) for any $x< y$ and $\delta>0$ such that $x+\delta, y+\delta \in [0,1]$ \cite[Sec.~1.4]{niculescu2018}.}
we have
\begin{align}\label{coi:n2}
    f(u_1)-f(x_1) = f(x_1+u_1-x_1) - f(x_1) < f(u_2+u_1-x_1)-f(u_2),
\end{align}
as $x_1 < u_1 < u_2.$
From  $\alpha \le 0.5$ and $\rK_{\alpha}(\bu) = \rK_{\alpha}(\bx)$, we get $u_1-x_1\le (1-\alpha)(u_1-x_1)/\alpha = x_2-u_2$.
Since $f$ is increasing, we obtain from \eqref{coi:n2} that $f(u_1)-f(x_1) < f(u_2+x_2-u_2)-f(u_2) = f(x_2)-f(u_2)$.
This proves \eqref{coi:n1}.

By the continuity and monotonicity of $f$ and  \eqref{coi:n1}, there is $\widehat{u}_1\in (u_1, u_2)$ such that 
$f(\widehat{u}_1) + f(u_2 ) < f(x_1) + f(x_2)$.  
Moreover, $\rK_{\alpha}(\widehat{\bu}) > \rK_{\alpha}(\bx)$, where $\widehat{\bu} = [\widehat{u}_1, u_2]$. 
Consequently, $\widehat{\bu} \prec_{\rA, \rB} \bx$ and $\bx \prec_{(\alpha, \beta)} \widehat{\bu}$.
\end{proof}

\section{Conclusion}

We have continued the line of research on admissibility of aggregation function pairs, building on the work of Bustince et al.~\cite{bustince2013}.
We have introduced several new examples of such pairs in some important classes of AFs, including quasi-linear means, Archimedean $t$-norms, and Archimedean $t$-conorms. 
For the latter families, we have obtained results that clearly determine whether a given pair consisting of an Archimedean $t$-norm and an Archimedean $t$-conorm is admissible. 
Similarly, in the case of quasi-linear mean pairs with identical weights, a complete characterization has been provided (see Corollary~\ref{cor:full}). However, identifying all admissible pairs of the form $(\rK^f_{w_1}, \rK^g_{w_2})$ with $w_1 \neq w_2$ remains an open problem.

Based on the constructions of strictly Schur-convex functions, we have identified the admissible order generated by an admissible pair of AFs, which does not coincide with any $(\alpha,\beta)$-order.
We propose, for future research, the characterization of all admissible pairs $(\rA, \rB)$ of AFs  for which $\leqa_{\rA,\rB}\,=\,\leqa_{(\alpha,\beta)}$ for some $\alpha\neq \beta$.

%%
%\section*{Declaration of competing interest}
%The authors declare that they have no known competing financial interests or personal relationships that could have appeared to influence the work reported in this paper.

%\section*{Data availability}
%No data was used for the research described in the article.

%%\section*{Acknowledgments}
%%Authors would like to express their sincere thanks and gratitude to the editor and  anonymous reviewers for their suggestions to improve the article.

\bibliographystyle{acm}
\bibliography{biblio}

@book{alsina2006,
  author    = {C. Alsina and M. J. Frank and B. Schweizer},
  title     = {Associative Functions: Triangular Norms and Copulas},
  publisher = {World Scientific},
  year      = {2006}
}

@article{asiain2018,
  author  = {M. J. Asiain and H. Bustince and R. Mesiar and A. Koles{\'a}rov{\'a} and Z. Tak{\'a}{\v c}},
  title   = {Negations with respect to admissible orders in the interval-valued fuzzy set theory},
  journal = {IEEE Transactions on Fuzzy Systems},
  volume  = {26},
  year    = {2018},
  pages   = {556--568}
}

@article{asmus2022,
  author  = {T. C. Asmus and J. A. Sanz and G. P. Dimuro and B. Bedregal and J. Fern\'{a}ndez and H. Bustince},
  title   = {N-dimensional admissibly ordered interval-valued overlap functions and its influence in interval-valued fuzzy-rule-based classification systems},
  journal = {IEEE Transactions on Fuzzy Systems},
  volume  = {30},
  year    = {2022},
  pages   = {1060--1072}
}

@article{bentkowska2015,
  author  = {U. Bentkowska and H. Bustince and A. Jurio and M. Pagola and B. Pękala},
  title   = {Decision making with an interval-valued fuzzy preference relation and admissible orders},
  journal = {Applied Soft Computing},
  volume  = {35},
  year    = {2015},
  pages   = {792--801}
}

@article{BJK21,
  author  = {M. Boczek and L. Jin and M. Kaluszka},
  title   = {Interval-valued seminormed fuzzy operators based on admissible orders},
  journal = {Information Sciences},
  volume  = {574},
  year    = {2021},
  pages   = {96--110}
}

@article{BJK22,
  author  = {M. Boczek and L. Jin and M. Kaluszka},
  title   = {The interval-valued {C}hoquet-{S}ugeno-like operator as a tool for aggregation of interval-valued functions},
  journal = {Fuzzy Sets and Systems},
  volume  = {448},
  year    = {2022},
  pages   = {35--48}
}

@article{BKL25,
  author  = {M. Boczek and M. Kaluszka and J. {\L}ompie{\'s}},
  title   = {On monotonicity of the copula-based interval-valued aggregation function},
  journal = {Fuzzy Sets and Systems},
  volume  = {502},
  year    = {2025},
  pages   = {109229}
}

@article{bustince2009,
  author  = {H. Bustince and E. Barrenechea and M. Pagola and J. Fernandez}, 
  title   = {Interval-valued fuzzy sets constructed from matrices: {A}pplication to edge detection},
  journal = {Fuzzy Sets and Systems},
  volume  = {160},
  year    = {2009},
  pages   = {1819--1840}
}

@article{bustince2013,
  author  = {H. Bustince and J. Fernandez and A. Koles{\'a}rov{\'a} and R. Mesiar},
  title   = {Generation of linear orders for intervals by means of aggregation functions},
  journal = {Fuzzy Sets and Systems},
  volume  = {220},
  year    = {2013},
  pages   = {69--77}
}

@article{bustince2013b,
  author={Bustince, Humberto and Galar, Mikel and Bedregal, Benjamin and Koles{\'a}rov{\'a}, Anna and Mesiar, Radko},
  title   = {A~new approach to interval-valued {C}hoquet integrals and the problem of ordering in interval-valued fuzzy set applications},
  journal = {IEEE Transactions on Fuzzy Systems},
  volume  = {21},
  year    = {2013},
  pages   = {1150--1162},
  publisher={IEEE}
}

@article{bustince2020,
  author  = {H. Bustince and C. Marco-Detchart and J. Fernandez and C. Wagner and J. M. Garibaldi and Z. Tak{\'a}{\v c}},
  title   = {Similarity between interval-valued fuzzy sets taking into account the width of the intervals and admissible orders},
  journal = {Fuzzy Sets and Systems},
  volume  = {390},
  year    = {2020},
  pages   = {23--47}
}

@article{daroczy1987,
  author  = {Z. Dar\'{o}czy and Z. P\'{a}les},
  title   = {Convexity with given infinite weight sequences},
  journal = {Stochastica},
  volume  = {11},
  year    = {1987},
  pages   = {5--12}
}

@article{derrac2016,
  author  = {J. Derrac and F. Chiclana and S. Garc{\'i}a and F. Herrera},
  title   = {Evolutionary fuzzy $k$-nearest neighbors algorithm using interval-valued fuzzy sets},
  journal = {Information Sciences},
  volume  = {329},
  year    = {2016},
  pages   = {144--163}
}

@article{miguel2016,
  title={Construction of admissible linear orders for interval-valued {A}tanassov intuitionistic fuzzy sets with an application to decision making},
  author={De Miguel, Laura and Bustince, Humberto and Fern{\'a}ndez, Javier and Indur{\'a}in, Esteban and Koles{\'a}rov{\'a}, Anna and Mesiar, Radko},
  journal={Information Fusion},
  volume={27},
  pages={189--197},
  year={2016},
  publisher={Elsevier}
}

@book{grabisch2009,
  author    = {M. Grabisch and J. Marichal and R. Mesiar and E. Pap},
  title     = {Aggregation Functions},
  publisher = {Cambridge University Press},
  year      = {2009},
  series    = {Encyclopedia of Mathematics and its Applications}
}

@article{gupta2023,
  author  = {V. K. Gupta and S. Massanet and N. R. Vemuri},
  title   = {Novel construction methods of interval-valued fuzzy negations and aggregation functions based on admissible orders},
  journal = {Fuzzy Sets and Systems},
  volume  = {473},
  year    = {2023},
  pages   = {108722}
}

@article{he2023,
  author  = {X. X. He and Y. F. Li and B. Yang},
  title   = {Interval-valued fuzzy logical connectives with respect to admissible orders},
  journal = {Iranian Journal of Fuzzy Systems},
  volume  = {20},
  year    = {2023},
  pages   = {1--19}
}

@book{klement2000,
  author    = {E. P. Klement and R. Mesiar and E. Pap},
  title     = {Triangular Norms},
  publisher = {Kluwer},
  address   = {Dordrecht},
  year      = {2000}
}

@book{nguyen2006,
  author    = {H. T. Nguyen and E. Walker},
  title     = {A First Course in Fuzzy Logic},
  publisher = {Chapman \& Hall/CRC},
  year      = {2006}
}

@book{niculescu2018,
  author    = {C. P. Niculescu and L. E. Persson},
  title     = {Convex Functions and Their Applications: A Contemporary Approach},
  publisher = {Springer},
  year      = {2018}
}

@book{marshall2011,
  author    = {I. Olkin and A. W. Marshall},
  title     = {Inequalities: Theory of Majorization and Its Applications},
  publisher = {Springer},
  year      = {2011}
}

@article{milfont2021,
  author={Milfont, Thadeu and Bedregal, Benjamin and Mezzomo, Ivan},
  title={Generation of admissible orders on n-dimensional fuzzy set $\mathrm{L}_n ([0, 1])$},
  journal={Information Sciences},
  volume={581},
  pages={856--875},
  year={2021},
  publisher={Elsevier}
}

@book{pekala2019,
  author    = {B. P\k{e}kala},
  title     = {Uncertainty Data in Interval-Valued Fuzzy Set Theory: Properties, Algorithms and Applications},
  publisher = {Springer},
  series    = {Studies in Fuzziness and Soft Computing},
  volume    = {367},
  year      = {2019}
}

@article{santana2020,
  author  = {F. Santana and B. Bedregal and P. Viana and H. Bustince},
  title   = {On admissible orders over closed subintervals of [0, 1]},
  journal = {Fuzzy Sets and Systems},
  volume  = {399},
  year    = {2020},
  pages   = {44--54}
}

@article{sussner2024,
  title={Construction of ${K}_{\alpha}$-orders including admissible ones on classes of discrete intervals},
  author={Sussner, Peter and Carazas, Lisbeth Corbacho},
  journal={Fuzzy Sets and Systems},
  volume={480},
  pages={108857},
  year={2024},
  publisher={Elsevier}
}

@article{sussner2025,
  title={On Counting and Constructing All Admissible Orders of the Non-empty Intervals in Any Finite Chain},
  author={Sussner, Peter and Vicentin, Felipe Scherer},
  journal={Fuzzy Sets and Systems},
  pages={109372},
  year={2025},
  publisher={Elsevier}
}

@article{takac2022,
  author  = {Z. Tak\'{a}\v{c} and M. Uriz and M. Galar and D. Paternain and H. Bustince},
  title   = {Discrete {IV} $d_{{G}}$-{C}hoquet integrals with respect to admissible orders},
  journal = {Fuzzy Sets and Systems},
  volume  = {441},
  year    = {2022},
  pages   = {169--195}
}

@article{wu2024,
  author  = {X. Wu and S.-M. Chen and X. Zhang},
  title   = {Generated admissible orders for intervals by matrices and continuous functions},
  journal = {Information Sciences},
  volume  = {659},
  year    = {2024},
  pages   = {120051}
}

@article{xu2006,
  author  = {Z. Xu and R. R. Yager},
  title   = {Some geometric aggregation operators based on intuitionistic fuzzy sets},
  journal = {International Journal of General Systems},
  volume  = {35},
  year    = {2006},
  pages   = {417--433}
}

@article{zapata2017,
  author  = {H. Zapata and H. Bustince and S. Montes and B. Bedregal and G. P. Dimuro and Z. Tak\'{a}\v{c} and M. Baczy\'{n}ski and J. Fernandez},
  title   = {Interval-valued implications and interval-valued strong equality index with admissible orders},
  journal = {International Journal of Approximate Reasoning},
  volume  = {88},
  year    = {2017},
  pages   = {91--109}
}

@article{zumelzu2022,
  title={Admissible orders on fuzzy numbers},
  author={Zumelzu, Nicol{\'a}s and Bedregal, Benjam{\'\i}n and Mansilla, Edmundo and Bustince, Humberto and D{\'\i}az, Roberto},
  journal={IEEE Transactions on Fuzzy Systems},
  volume={30},
  number={11},
  pages={4788--4799},
  year={2022},
  publisher={IEEE}
}

@article{zhang2025,
  author={Zhang, Wei},
  title={An equivalent characterization for admissible orders on n-dimensional intervals generated by matrices},
  journal={International Journal of Approximate Reasoning},
  pages={109438},
  year={2025},
  publisher={Elsevier}
}

@article{yzhao2024,
  author  = {Y. Zhao and L. Hua-Wen},
  title   = {Interval {R}-{S}heffer strokes and interval fuzzy {S}heffer strokes endowed with admissible orders},
  journal = {International Journal of Approximate Reasoning},
  volume  = {166},
  year    = {2024},
  pages   = {109120}
}

\end{document}